\documentclass[12pt]{article}
\usepackage{amsmath,amsfonts,amssymb,amsthm}
\usepackage{bbm,mathrsfs,bm,mathtools}
\title{Anisotropic area measures of convex bodies}
\author{Rolf Schneider}
\date{}
\newcommand{\Sn}{{\mathbb S}^{n-1}}

\newcommand{\R}{{\mathbb R}}

\newcommand{\K}{{\mathcal K}}

\newcommand{\N}{{\mathbb N}}

\newcommand{\Ha}{\mathcal{H}}

\newcommand{\B}{\mathcal{B}}
\newcommand{\D}{{\rm d}}

\newtheorem{theorem}{Theorem}
\newtheorem{lemma}{Lemma}
\newtheorem{definition}{Definition}

\begin{document}
\maketitle

\begin{abstract}
{Motivated by the relative differential geometry, where the Euclidean normal vector of hypersurfaces is generalized by a relative normalization, we introduce anisotropic area measures of convex bodies, constructed with respect to a gauge body. Together with the anisotropic curvature measures, they are special cases of the newly introduced anisotropic support measures. We show that a convex body in $\R^n$, for which the anisotropic area measure of some order $k\in\{0,\dots,n-2\}$ is proportional to the area measure of order $n-1$, must be a $k$-tangential body of the gauge body.}\\[2mm]
\end{abstract}

\section{Introduction}\label{sec1}

It is a classical result of global differential geometry that a compact hypersurface (sufficiently smooth and without boundary), which is embedded in Euclidean space $\R^n$ and on which some elementary symmetric function of the principal curvatures is constant, must be a sphere. While versions for convex hypersurfaces go back to Liebmann \cite{Lie99} and S\"uss \cite{Sus29}, contributions with increasing generality are due to Hsiung \cite{Hsi54}, Aleksandrov \cite{Ale56}, Reilly \cite{Rei77}, until the full result was proved by Ros \cite{Ros87}; a new proof was given by Montiel and Ros \cite{MR91}.

The restriction to convex hypersurfaces has the advantage that one can prove a corresponding results without differentiability assumptions. For a convex body $K$ (a compact convex set with interior points) in $\R^n$ one can define curvature measures $C_k(K,\cdot)$, $k=0,\dots,n-1$, (see \cite[Sect. 4.2]{Sch14}) that depend weakly continuously on $K$ and satisfy
$$ C_k(K,\beta) = \int_\beta H_{n-1-k}\,\D \Ha^{n-1}$$
for Borel sets $\beta$ in the boundary of $K$, if the latter is of class $C^2$. Here $H_m$ is the $m$-th normalized elementary symmetric function of the principal curvatures and $\Ha^{n-1}$ is the $(n-1)$-dimensional Hausdorff measure. Since $C_{n-1}(K,\beta) = \Ha^{n-1}(\beta)$ for $\beta$ as above, the assumption that $H_{n-1-k}$ is constant can be generalized to the assumption 
\begin{equation}\label{1.1N}
C_k(K,\cdot) =c C_{n-1}(K,\cdot)
\end{equation}
with some constant $c$, where $k\in \{0,\dots,n-2\}$. Under this assumption, it was proved in \cite{Sch79} that $K$ must be a ball. More generally, Kohlmann \cite{Koh98} showed (with a very different proof) that $K$ must be a ball if $\sum_{r=0}^{n-2} \lambda_rC_r(K,\cdot)=C_{n-1}(K,\cdot)$ with $\lambda_0,\dots,\lambda_{n-2}\ge 0$.

If $K$ is a convex body of class $C^2_+$, then the principal curvatures $k_1,\dots,k_{n-1}$ of its boundary are all positive, and one can define the principal radii $r_1,\dots,r_{n-1}$ by $r_i=1/k_i$. Let $s_k$ be the $k$-th normalized elementary symmetric function of the principal radii, considered as a function of the outer unit normal vector. Similar to the curvature measures, one can define area measures $S_k(K,\cdot)$, $k=0,\dots,n-1$, (see \cite[Sect. 4.2]{Sch14}) that depend weakly continuously on $K$ and satisfy
$$ S_k(K,\omega)= \int_\omega s_k\,\D\Ha^{n-1} $$
for Borel sets $\omega$ in the unit sphere $\Sn$ of $\R^n$, if $K$ is of class $C^2_+$. Under sufficient differentiability assumptions, the condition $H_{n-1-k}=c$ is equivalent to $s_k=cs_{n-1}$, which without differentiability assumptions can be replaced by
\begin{equation}\label{1.2N}
S_k(K,\cdot)= c S_{n-1}(K,\cdot).
\end{equation}
Interestingly, in contrast to (\ref{1.1N}), the equation (\ref{1.2N}) does not characterize balls. It was proved in \cite{Sch78}  that (\ref{1.2N}) holds if and only if $K$ is a $k$-tangential body of a ball. The latter means that $K$ contains a ball $B$ and that every supporting hyperplane of $K$ that does not support $B$ contains only $(k-1)$-singular points of $K$ (see \cite[Thm. 2.2.10]{Sch14}), that is, points at which there exist $n-k+1$ linearly independent normal vectors. For example, the convex hull of $B$ and a point outside is a $1$-tangential body of $B$.

In relative differential geometry, the unit normal vector of a hypersurface is replaced by a `relative normalization', in such a way that the gauge hypersurface representing the relative normal and the original hypersurface have parallel tangent hyperplanes at corresponding points. The relative normal is then used to define relative curvatures and to carry over other constructions of classical differential geometry. For introductions to relative differential geometry, with references to the early history, we refer to Bonnesen and Fenchel \cite[Sect. 38]{BF34} and P.A. and A.P. Schirokow \cite[Sect. VIII]{Sch63}. Presentations of basic aspects, preparing different applications, are also given by Simon \cite[§1]{Sim67} and Schneider \cite[Sect. 2]{Sch67}, and in a different style by Scheuer and Zhang \cite[Sect. 2]{SZ25}.

It seems that a revival of relative differential geometry began with Reilly \cite{Rei76}. In the introduction, the author wrote: ``The primary objective of this paper is to interpret certain theorems from the theory of nonlinear partial differential equations in terms of the relative differential geometry of nonparametric hypersurfaces.'' The theorem on closed hypersurfaces with a constant curvature function was extended to relative curvatures by He, Li, Ma and Ge \cite{HLMG09}. More contributions, emphasizing the anisotropic character, are due to He and Li \cite{HL08}, Xia \cite{Xia13}, Scheuer and Zhang \cite{SZ25}. Anisotropic curvature measures of convex bodies were introduced, and the result of Kohlmann mentioned above was extended to this setting, by Andrews, Lei, Wei and Xiong \cite{ALWX21}. As a counterpart, the purpose of the following is to introduce anisotropic area measures, depending on a gauge body $E$, and to characterize the tangential bodies of $E$, thus generalizing the result of \cite{Sch78}.

We formulate our result in the next section, after a few preliminaries.

\section{Anisotropic support measures}\label{sec2}

We work in $\R^n$ ($n\ge 2$), with scalar product $\langle\cdot\,,\cdot\rangle$ and norm $\|\cdot\|$, unit sphere $\Sn$ and unit ball $B^n$ (with centers at the origin $o$). For subsets of $\R^n$, we denote by ${\rm int}$ the interior and by ${\rm bd}$ the boundary. Let $\B(T)$ be the $\sigma$-algebra of Borel sets of a topological space $T$.

Anisotropic area measures are easily introduced if one is familiar with mixed area measures of convex bodies, as we will assume; we refer to \cite[Sect. 5.1]{Sch14}.
By $\K^n$ we denote the set of convex bodies in $\R^n$. The mixed area measure of $K_1,\dots,K_{n-1}\in\K^n$ is denoted by $S(K_1,\dots,K_{n-1},\cdot)$, It is a finite Borel measure on the unit sphere $\Sn$, and is connected to the mixed volume $V(K_1,\dots,K_n)$ of $K_1,\dots,K_n\in\K^n$ by
\begin{equation}\label{2.1N}
V(K_1,\dots,K_n)= \frac{1}{n}\int_{\Sn} h_{K_1}(u)\,S(K_2,\dots,K_n,\D u),
\end{equation}
where
$$ h_K(u) =\max\{\langle x,u\rangle:x\in K\} \quad\mbox{for }u\in\Sn$$
defines the support function of $K$ (here only needed on $\Sn$).

The principal idea of relative differential geometry is to replace the unit ball $B^n$ by a general gauge body $E$, and we do the same for general convex bodies. Let $E\in\K^n$ be given, with the property that $o\in{\rm int}\,E$. In the following, $E$ is fixed and will, therefore, often not be mentioned. In this section, we assume only that $E$ is regular and strictly convex. Therefore, at any boundary point $x$ of $E$ there is a unique outer unit normal vector $u_E(x)$, and any given unit vector $u$ is attained as an outer normal vector of $E$ at a unique boundary point $x_E(u)$ of $E$. 

We have thus defined a map $u_E: {\rm bd}\,E\to\Sn$, and since $E$ is regular and strictly convex, this is a homeomorphism from ${\rm bd}\,E$ to $\Sn$. We use it in the following definition. For $\alpha\subset\R^n$ we write $u_E(\alpha)=\{u_E(x): x\in\alpha\cap {\rm bd}\,E\}$.
\begin{definition}\label{D2.1N}
For $K\in\K^n$ and $k\in\{0,\dots,n-1\}$, the $k$-th anisotropic area measure of $K$ is defined by
$$ S_k^E(K,\alpha):= S(K[k],E[n-1-k],u_E(\alpha))\quad\mbox{for }\alpha\in \B(\R^n).$$
\end{definition}

Here $L[m]$ indicates that the argument $L$ appears $m$ times. By its definition, $S_k^E(K,\cdot)$ is a finite Borel measure, concentrated on the boundary of $E$.

We will show in the next section that $S_k^E(K,\cdot)$ is in fact the correct replacement for the $k$-th elementary symmetric function of the relative radii.

We formulate our main result.
\begin{theorem}\label{T2.1N}
Let $K\in\K^n$ and $k\in\{0,\dots,n-2\}$. The relation
$$
S_k^E(K,\cdot) = cS_{n-1}^E(K,\cdot)
$$
with a constant $c$ holds if and only if $K$ is homothetic to a $k$-tangential body of $E$.
\end{theorem}

For tangential bodies, we refer to \cite[Sect. 2.2]{Sch14}. Theorem \ref{T2.1N} will be proved in Section \ref{sec4}.

We add a remark on the Minkowski problem in its measure version. The relative Minkowski problem would ask: Given a finite Borel measure $\varphi$ on the boundary of the gauge body $E$, what are the necessary and sufficient conditions on $\varphi$ in order that there is a convex body $K\in\K^n$ such that $S_{n-1}^E(K,\cdot)=\varphi$? But this is not really a new problem. Since
$$ S_{n-1}^E(K,\alpha) = S(K[n-1],u_E(\alpha)) = S_{n-1}(K,u_E(\alpha)) \quad\mbox{for }\alpha\in\B({\rm bd}\,E),$$
the gauge body enters only via $u_E(\alpha)$. Hence, if we define $\psi(\omega):= \varphi(u_E^{-1}(\omega))$ for $\omega\in\B(\Sn)$, then the solution of the classical Minkowski problem for $\psi$ also solves the relative Minkowski problem for $\varphi$. It is obvious how to carry over the conditions from $\psi$ to $\varphi$.

Before entering into the proof of Theorem \ref{T2.1N}, we use the rest of this section to point to a common generalization of anisotropic curvature and area measures, in the form of anisotropic support measures. We modify the approach in \cite{Sch14}.

Let $K\in\K^n$. Let $x\in{\rm bd}\,K$, and let $u$ be an outer unit normal vector of $K$ at $x$. The vector $x_E(u)$ is called a {\em relative normal vector} of $K$ at $x$. A {\em relative support element} of $K$ is a pair $(x,y)$ where $x\in{\rm bd}\,K$ and $y$ is a relative normal vector of $K$ at $x$. We denote by $\Sigma^E(K)$ the set of all relative support elements of $K$, and we define $\Sigma:= \R^n\times\R^n$.

Let $x\in \R^n\setminus K$. Let $s>0$ be the largest number such that $(-sE+x)\cap {\rm int}\,K =\emptyset$. We set
$$ d^E(K,x):= s$$
and call the number $d^E(K,x)$ the {\em $E$-distance} of $x$ from $K$. We set $d^E(K,x)=0$ if $x\in K$. For $r\ge 0$ we then have
\begin{equation}\label{3.1}
d^E(K,x)\le r \Leftrightarrow (-rE+x)\cap K\not=\emptyset.
\end{equation}

There is a unique hyperplane $H$ that separates $K$ and $-sE+x$, and $K$ and $-sE+x$ have a unique common boundary point $p$. We set
$$ p^E(K,x):= p.$$
The mapping $p^E(K,\cdot)$ is called the {\em $E$-projection} to $K$.

Further, the suitably oriented unit normal vector $u$ of $H$ is an outer normal vector of $K$ at $p$, and $-u$ is an outer normal vector of $-sE+x$ at $p$. We set
$$ y^E(K,x):= x_E(u).$$
The pair $(p^E(K,x),y^E(K,x))$ is a relative support element of $K$. We note that (for fixed $K$) the point $p^E(K,x)$ and the vector $y^E(K,x)$ are uniquely determined by $x\in\R^n\setminus K$. 

\begin{lemma}\label{L1}
For $\rho>0$ and $x\in\R^n\setminus K$,
$$ d^E(K,x) = \rho \Leftrightarrow x\in {\rm bd}(K+\rho E).$$
\end{lemma}

\begin{proof}
Let $x\in {\rm bd}(K+\rho E)$. Then there exists $p\in K$ such that $x\in p+\rho E$, and necessarily $x\in {\rm bd}(p+\rho E)$ (otherwise, $x\in {\rm int}(K+\rho E)$). There is a common supporting hyperplane $H$ of $K+\rho E$ and $p+\rho E$ at $x$, and it is unique since $E$ is regular. The set $x+\rho (-E)$ is the image of $p+\rho E$ under reflection in the midpoint of the segment $[x,p]$. Let $u$ be the unit normal vector of $H$ that is an outer normal vector of $K+\rho E$. Since the support functions satisfy $h_{K+\rho E}(u)= h_K(u)+h_{\rho E}(u)$, the hyperplane parallel to $H$ through $p$ supports $K$. Therefore, $(x+\rho(- E))\cap{\rm int}\,K=\emptyset$. Now it follows that $d^E(K,x)=\rho$.

Conversely, let $d^E(K,x)=\rho$. Then $x+\rho(-E)$ touches $K$ in a point $p$, and we have $x\in p+\rho E\subset K+\rho E$. If $x\in{\rm int}(K+\rho E)$, then $x\in{\rm bd}(K+\rho' E)$ for some $\rho'<\rho$, which, as shown above, would imply $d^E(x,K) =\rho'<\rho$, a contradiction.
\end{proof}

It follows that
\begin {equation}\label{A4}
\{x\in\R^n: 0< d^E(K,x)\le \rho\}= (K+\rho E)\setminus K.
\end{equation}

We need some continuity results, from which then measurability follows. First we remark the following. Let $r_0(E)$ be the radius of the largest ball with center $o$ contained in $E$. Then $r_0(E)B^n\subseteq E$ and hence
\begin{equation}\label{A5}
E+\rho B^n \subseteq\left(1+\frac{\rho}{r_0(E)}\right)E
\end{equation}
for $\rho>0$. For $x,\overline x\in\R^n\setminus K$ we state that
\begin{equation}\label{A7}
|d^E(K,x) - d^E(K,\overline x)|\le\frac{1}{r_0(E)}\|x-\overline x\|.
\end{equation}
For the proof, write $\|x-\overline x\|=a$. From (\ref{A5}) (with $\rho=a/d^E(K,x)$) we get
\begin{equation}\label{3.2}
-d^E(K,x)E+aB^n \subseteq -\left(d^E(K,x)+\frac{a}{r_0(E)}\right)E.
\end{equation}
By (\ref{3.1}), there is a point $y\in (-d^E(K,x)E+x)\cap K$. This gives
$$ y\in -d^E(K,x)+\overline x -\overline x +x \subset -d^E(K,x)+\overline x +a B^n\subset -\left(d^E(K,x)+\frac{a}{r_0(E)}\right)E+\overline x$$
by (\ref{3.2}), and since $y\in K$, it follows from (\ref{3.1}) that
$$ d^E(K,\overline x) \le d^E(K,x)+\frac{a}{r_0(E)}.$$
Together with interchanging $x$ and $\overline x$ this gives (\ref{A7}).

For $x_j\in\R^n\setminus K$, $j\in\N_0$, with $x_j\to x_0$ as $j\to\infty$, we have 
$$
d^E(K,x_j)\to d^E(K,x_0),\quad p^E(K,x_j)\to p^E(K,x_0),\quad u^E(K,x_j)\to u^E(K,x_0).
$$
The first relation follows from (\ref{A7}).
For the other two relations we note that from $x_j\to x_0$ and $d^E(K,x_j)\to d^E(K,x_0)$ it can be deduced that
$$ -d^E(K,x_j)E+x_j \to -d^E(K,x_0)E+x_0$$
in the sense of the Hausdorff metric. From this, the assertions follow immediately. 

To define now local parallel sets, we can argue similarly as in \cite[Sect. 4.1]{Sch14}. Let $\rho>0$. The map
$$ \begin{array}{cccc}
f^E_{\rho}: & (K+\rho E)\setminus K & \to & \Sigma\\
& x & \mapsto & (p^E(K,x),y^E(K,x)) 
\end{array}$$
is continuous and hence measurable. Also the function $d^E(K,\cdot)$ is continuous. Let $\eta\in \B(\Sigma)$. The {\em relative local parallel set}
$$ M^E_\rho(K,\eta):= (f^E_\rho)^{-1}(\eta) =\{x\in\R^n: 0< d^E(K,x)\le\rho,\, (p^E(K,x),y^E(K,x))\in\eta\}$$
is measurable. By (\ref{A4}) we have
$$
M^E_\rho(K,\Sigma)= (K+\rho E)\setminus K.
$$

Different from the case $E=B^n$, we now do not consider the $n$-dimensional Hausdorff measure of the local parallel set $M_\rho^E(K,\eta)$, but the $(n-1)$-dimensional Hausdorff measure of the set
$$ \gamma(K,E,\rho,\eta):= M_\rho^E(K,\eta)\cap {\rm bd}(K+\rho E).$$
By $\omega(K,E,\eta)$ we denote the set of outer unit normal vectors of $K+\rho E$ at points of $\gamma(K,E,\rho,\eta)$, for $\rho>0$. We have $u\in \omega(K,E,\eta)$ if and only if there exists $(x,y)\in\eta\cap \Sigma^E(K)$ with $u= u_E(y)$. This implies that the set $\omega(K,E,\eta)$ is independent of $\rho$, and that $\{\omega(K,E,\eta):\eta\in \B(\Sigma)\}$ is a $\sigma$-algebra containing the closed sets in $\Sn$. For $\alpha,\beta\subset\R^n$ we have
$$ \omega(K,E,\beta\times\R^n) = {\bf u}_K(\beta),\qquad \omega(K,E,\R^n\times \alpha)= u_E(\alpha),$$
where ${\bf u}_K(\beta)$ denotes the set of outer unit normal vectors of $K$ at points of $\beta\cap{\rm bd}\,K$.

Since the push-forward of the $(n-1)$-dimensional Hausdorff measure $\Ha^{n-1}$, restricted to the boundary of a convex body $L$, under the Gauss map gives the surface area measure $S_{n-1}(L,\cdot)$ of $L$, we have
\begin{equation}\label{A12a} 
\Ha^{n-1}(\gamma(K,E,\rho,\eta))=S_{n-1}(K+\rho E,\omega(K,E,\eta)).
\end{equation}
It follows from \cite[(5.18)]{Sch14} that
\begin{equation}\label{A13}
S_{n-1}(K+\rho E,\cdot) = \sum_{m=0}^{n-1} \rho^{n-1-m}\binom{n-1}{m} S(K[m],E[n-1-m],\cdot).
\end{equation}
Now we are ready for the following definition.
\begin{definition}\label{D2.2N}
For $K\in\K^n$ and $k\in\{0,\dots,n-1\}$, the $k$-th anisotropic support measure of $K$ is defined by
\begin{equation}\label{A13a} 
\Theta_k^E(K,\eta):= S(K[k],E[n-1-k],\omega(K,E,\eta))\quad \mbox{for }\eta\in\B(\Sigma).
\end{equation}
\end{definition}
It follows from (\ref{A12a}), (\ref{A13}) and (\ref{A13a}) that
$$\Ha^{n-1}(M_\rho^E(K,\eta)\cap {\rm bd}(K+\rho E))= \sum_{m=0}^{n-1} \rho^{n-1-m} \binom{n-1}{m} \Theta_m^E(K,\eta).$$
In particular, if for $\alpha\in\B(\R^n)$ and $\rho>0$ we define
\begin{equation}\label{N1} 
B_\rho^E(K,\alpha) = \{x\in\R^n: d^E(K,x) = \rho,\, y^E(K,x)\in\alpha\},
\end{equation}
we have
\begin{equation}\label{A15}
\Ha^{n-1}(B_\rho^E(K,\alpha))= \sum_{m=0}^{n-1} \rho^{n-1-m} \binom{n-1}{m} S_m^E(K,\alpha).
\end{equation}

We can now also define the anisotropic curvature measures, by
$$ C_k^E(K,\beta)= \Theta_k^E(K,\beta\times\R^n)=S(K[k],E[n-1-k],{\bf u}_K(\beta)) \quad\mbox{for } \beta\in\B(\R^n).$$

\section{Comparison with relative differential geometry}\label{sec3}

Here we point out only the changes which are necessary in \cite[Sect. 2.5]{Sch14}. We assume in this section that the convex body $K$ and the gauge body $E$ are of class $C^2_+$. Then also for each $x\in{\rm bd}\,K$ there is a unique outer unit normal vector of $K$ at $x$, which we denote by $u_K(x)$. First we re-normalize the Euclidean unit normal vector of $K$ by
$$ \xi(x)= \frac{u_K(x)}{h_E(u_K(x))}\quad\mbox{for }x\in{\rm bd}\,K.$$

We assume that a local parametrization of class $C^2$ of a neighborhood of the considered point $x\in{\rm bd}\,K$ is given by $X:M\to{\rm bd}\,K$, where $M\subset \R^{n-1}$ is an open subset. We define a local parametrization of ${\rm bd}\,E$ by
$$ Y:= u_E^{-1}\circ u_K\circ X.$$
Then $E$ and $K$ have the same outer unit normal vector at points with the same parameters. The relative normal vector of $K$ at $X(z)$ is given by $Y(z)$. The vector function $\xi$ is locally parametrized by
$$ \Xi= \frac{u_K\circ X}{h_E(u_E\circ Y)}.$$

We use classical tensor notation (including the summation convention) and denote differentiation with respect to local parameters by indices. We have
$$ \langle \Xi,X_i\rangle =0, \quad \langle \Xi, Y\rangle =1.$$
Then we define symmetric tensors $G_{ij}$ and $B_{ij}$ by
$$ G_{ij} = \langle\Xi_i,X_j\rangle,\qquad B_{ij}=\langle \Xi_i, Y_j\rangle$$
(the symmetry of the latter follows from $\langle \Xi,Y_j\rangle =0$). We have chosen the signs in view of the fact that $\xi$ is an outer normal vector of $K$ and of $E$, so that the matrices $(G_{ij})$ and $(B_{ij})$ become positive definite (a proof can be found in \cite[p. 115]{Sch14}). 

We have the relation
\begin{equation}\label{3.1N}
X_i=G_{ij}B^{jr}Y_r,
\end{equation}
where $(B^{ij})$ is the matrix inverse to $(B_{ij})$. Starting with $X_i=c_i^{\hspace{1mm}r}Y_r$, we take the scalar product with $\Xi_j$ and obtain $G_{ij} = c_i^{\hspace{1mm}r}B_{rj}$ and hence $G_{ij} B^{jk} = c_i^{\hspace{1mm}r}B_{rj}B^{jk} = c_i^{\hspace{1mm}r}\delta_r^{\hspace{1mm}k} = c_i^{\hspace{1mm}k}$, which proves (\ref{3.1N}). (Here $\delta_i^{\hspace{1mm}k}=0$ for $i\not= k$ and $\delta_i^{\hspace{1mm}i}=1$.)

The {\em relative principal curvatures} are defined as the eigenvalues of $B_{ij}$ with respect to $G_{ij}$. Equivalently, they are the solutions $\lambda$ of the equation
$$ \det(B_{ij} -\lambda G_{ij})=0.$$
Therefore, the relative radii are the solutions $\lambda$ of the equation $\det(G_{ij}B^{jk}-\lambda \delta_i^{\hspace{1mm}k})=0$.

Let $\rho>0$. Defining $\{X(z)+\rho Y(z):z\in M\}=a^E_\rho(K,M)$, this is the image of $M$ under the injective map $z\mapsto X(z)+\rho Y(z)$, which has Jacobian 
$$ |X_1+\rho Y_1,\dots, X_{n-1}+\rho Y_{n-1}|$$ 
(here $|\cdot|$ denotes a determinant). Using (\ref{3.1N}), we get
\begin{eqnarray*}
&& |X_1+\rho Y_1,\dots, X_{n-1}+\rho Y_{n-1}|\\
&& = |(G_{1j}B^{jr}+\rho \delta_1^{\hspace{1mm}r})Y_r,\dots, (G_{(n-1)j}B^{jr}+\rho \delta_{n-1}^{\hspace{1mm}r})Y_r|\\
&& = \det(G_{ij}B^{jr} +\rho \delta_i^{\hspace{1mm}r})|Y_1,\dots,Y_{n-1}|.
\end{eqnarray*}
Since the relative radii $r_1^E,\dots, r_{n-1}^E$ are the eigenvalues of the matrix $(G_{ij}B^{jr})$, the matrix $(G_{ij}B^{jr}+\rho\delta_i^{\hspace{1mm}r})$ has the eigenvalues $r_1^E+\rho,\dots,r_{n-1}^E+\rho$, hence
$$ |X_1+\rho Y_1,\dots, X_{n-1}+\rho Y_{n-1}|=\pm \sum_{m=0}^{n-1} \rho^m\binom{n-1}{m} s_{n-1-m}^E\circ Y\sqrt{\det(h_{ij})},$$
where $h_{ij}=\langle Y_i,Y_j\rangle$ is the first fundamental form of ${\rm bd}\,E$. By integration, we get
$$ \Ha^{n-1}(a_\rho^E(K,M)) = \sum_{m=0}^{n-1} \rho^{n-1-m} \binom{n-1}{m} \int_{Y(M)} s_m^E\,\D\Ha^{n-1}.$$

If $\alpha\subset\R^n$ is an open set with $\alpha\cap{\rm bd}\,E=Y(M)$, then $a_\rho^E(K,M)=B_\rho^E(K,\alpha)$ (defined by (\ref{N1})) and therefore
$$ 
\Ha^{n-1}(B_\rho^E(K,\alpha)) = \sum_{m=0}^{n-1}\rho^{n-1-m}\binom{n-1}{m} \int_{\alpha\cap{\rm bd}\,E} s_m^E\,\D \Ha^{n-1}.
$$
Since this holds for open sets, it holds for arbitrary Borel sets $\alpha\in\B(\R^n)$. 

Comparison with (\ref{A15}) now shows that
$$
S_k^E(K,\alpha)= \int_{\alpha\cap {\rm bd}\, E} s_k^E \,\D \Ha^{n-1},
$$
which in the case $E= B^n$ reduces to \cite[(4.26)]{Sch14}. Thus, the measures $S_k^E(K,\cdot)$ are in fact the correct generalizations of the elementary symmetric functions of the relative radii.

\section{Proof of Theorem \ref{T2.1N}}\label{sec4}

For $K\in\K^n$ and $j\in\{0,\dots,n\}$ we use the notation
$$
V_j^E= V(K[j],E[n-j]).
$$

Now we assume again that $E$ has the properties as in Section \ref{sec2}, and we assume that $K\in\K^n$ and $k\in\{0,\dots,n-2\}$ satisfy
$$ S_k^E(K,\cdot)= cS_{n-1}^E(K,\cdot)$$
with some constant $c$. Replacing $K$ by a suitable homothet, we can assume that $c=1$. Then
\begin{equation}\label{A25a} 
S(K[k],E[n-1-k],\cdot)= S(K[n-1],\cdot),
\end{equation}
by Definition \ref{D2.1N} and the fact that $u_E$ is a homeomorphism.

By (\ref{A25a}) and (\ref{2.1N}) (where we choose $K$ or $E$ for $K_1$) we get 
\begin{equation}\label{A26} 
V_{k+1}^E = V_n^E,\qquad V_k^E= V_{n-1}^E.
\end{equation}
By the Aleksandrov--Fenchel inequalities for mixed volumes (e.g., \cite[Thm. 7.3.1]{Sch14}), we have
$$ \frac{V_{n-1}^E}{V_n^E} \ge  \frac{V_{n-2}^E}{V_{n-1}^E} \ge \dots \ge  \frac{V_k^E}{V_{k+1}^E}.$$
By (\ref{A26}), this holds with equality signs everywhere, in particular, $(V_{n-1}^E)^2 = V_{n-2}^EV_n^E$. By a result of Bol \cite{Bol43} (which is reproduced in \cite[Thm. 7.6.19]{Sch14}), $K$ is homothetic to an $(n-2)$-tangential body of $E$. Thus, after a translation we can assume that $K$ is an $(n-2)$-tangential body of $rE$, for some $r>0$. We state that
\begin{equation}\label{A27}
V_n^E= rV_{n-1}^E.
\end{equation}
In fact, by (\ref{2.1N}) and $S(K,\dots,K,\cdot)= u_K\sharp\Ha^{n-1}|_{{\rm bd}\,K}$ (where $u_K\sharp$ denotes the push-forward under the $\Ha^{n-1}$-almost everywhere on ${\rm bd}\,K$ defined Gauss map of $K$), assertion (\ref{A27}) is equivalent to
$$ \int_{{\rm bd}\,K} h_K\circ u_K\,\D\Ha^{n-1} = r\int_{{\rm bd}\,K} h_E\circ u_K\,\D\Ha^{n-1}.$$
Since $K$ is a tangential body of $rE$, we have $h_K(u)=rh_E(u)$ whenever $u$ is an outer unit normal vector at a regular point of $K$, and hence $(h_K\circ u_K)(x) = (h_E\circ u_K)(x)$ holds for $\Ha^{n-1}$-almost all $x\in {\rm bd}\,K$. This proves (\ref{A27}). 

We have proved that
$$ r =\frac{V_{n}^E}{V_{n-1}^E} =  \frac{V_{n-1}^E}{V_{n-2}^E} = \dots =  \frac{V_{k+1}^E}{V_{k}^E},$$
from which it follows that $V_n^E= r^{n-p}V_p^E$ for $p=k,\dots,n-1$. Together with (\ref{A26}), this shows that $r=1$ and then that $V_n^E= \dots = V_k^E$. Since $K$ contains a translate of $E$, the latter implies, as shown by Favard \cite[p. 273]{Fav33}, that $K$ is a $k$-tangential body of $E$ (or see \cite[Thm. 7.6.17]{Sch14} with $L=E$ and $n-p=k$, where (7.149) is equivalent to $V_k^E= V_{k+1}^E$).

Conversely, suppose that $K$ is a $k$-tangential body of $E$. Since Favard gives a necessary and sufficient condition (or see \cite[Thm. 7.6.17]{Sch14} ), this implies that $V_n^E= \dots = V_k^E$. From this, we deduce that 
\begin{equation}\label{4.2}
(V_j^E)^2 = V_{j-1}^EV_{j+1}^E\quad\mbox{for }j=k+1,\dots,n-1.
\end{equation}
In order to apply \cite[Thm. 7.4.6]{Sch14}, we set there on p. 406:
$$ m=n-k,\quad K_0=K,\, K_1=E,\quad {\mathscr C}=(K[k]).$$
Then the connection with our present notation is given by $V_{(i)}= V_{n-i}^E$, $S_{(i)} = S_{n-1-i}^E(K,\cdot)$, so that the condition (b) of \cite[Thm. 7.4.6]{Sch14} is satisfied, according to (\ref{4.2}). From (d) of that theorem it now follows that the measures $S_{n-1}^E(K,\cdot)$ and $S_k^E(K,\cdot)$ are proportional.
\hspace*{\fill}$\Box$

\noindent Author's address:\\[2mm]
Rolf Schneider\\Mathematisches Institut, Albert--Ludwigs-Universit{\"a}t\\D-79104 Freiburg i.~Br., Germany\\E-mail: rolf.schneider@math.uni-freiburg.de

\end{document}